\documentclass[12pt]{article}
\usepackage{graphicx}
\usepackage{hyperref}
\usepackage{listings}
\usepackage{amsmath}
\usepackage{amssymb}
\usepackage{amsfonts}
\usepackage{amsthm}
\usepackage[numbers]{natbib}
\usepackage{tikz}
\usetikzlibrary{shapes.geometric, arrows}
\usepackage{tikz,fullpage}
\usepackage{tkz-berge}
\usepackage{caption}
\usepackage{placeins}

\newtheorem{theorem}{Theorem}

\newtheorem{lemma}{Lemma}

\theoremstyle{remark}
\newtheorem{remark}{Remark}

\theoremstyle{definition}
\newtheorem{definition}{Definition}

\title{Counts of Hamiltonian Paths and Cycles for Complete Multipartite Graphs}
\author{Harry Hylock, Karl Michael Schmidt, Matthew C. Lettington}
\date{}

\begin{document}

\maketitle

\begin{abstract}
We prove counting formulae for the number of directed Hamiltonian paths and
directed Hamiltonian cycles in any complete multipartite graph.
The method is based on a collapsing technique relating general path graphs
with coloured vertices to legally coloured path graphs.
\end{abstract}

\section{Enumerating Hamiltonian Paths and Cycles}

A {\it complete multipartite graph\/} $K_{V_1,V_2,...,V_m}$ is a graph
with vertex set $V$ that splits into disjoint partite subsets $V_1, V_2, ...,V_m$ (so the cardinalities satisfy $|V| = \sum_{j=1}^m |V_j|$) such that any two
vertices are connected by an edge if and only if the vertices lie in different
partite subsets. For example, here is the graph $K_{\{v_1\}, \{v_2, v_3\}, \{v_4\}}$
with the partite subsets indicated by different shapes.

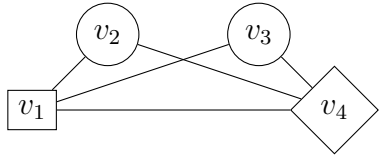
\begin{figure}[h]
    \centering
    \begin{tikzpicture}[scale=0.5]
        \node[rectangle,draw=black] (1) at (-1,-1) {$v_1$};
        \node[circle,draw=black] (2) at (1,1) {$v_2$};
        \node[circle,draw=black] (3) at (5,1) {$v_3$};
        \node[diamond,draw=black] (4) at (7,-1) {$v_4$};
        \draw[color=black] (1) -- (2);
        \draw[color=black] (1) -- (3);
        \draw[color=black] (1) -- (4);
        \draw[color=black] (2) -- (4);
        \draw[color=black] (3) -- (4);
    \end{tikzpicture}
\end{figure}

\noindent
We shall count the total number of directed Hamiltonian paths and cycles that exist in such graphs.
A formula for the count of Hamiltonian cycles was recently obtained by a
different method in \cite{MSmirnov}.
\begin{definition}
A {\it Hamiltonian path\/} in a graph is a subgraph that is a path graph and
contains all vertices of the original graph.
A {\it Hamiltonian cycle\/} in a graph is a subgraph that is a cycle graph and
contains all vertices of the original graph.
A {\it directed Hamiltonian path [cycle]\/} is a Hamiltonian path [cycle] with directed edges
such that each vertex is the start of one edge and the end of another.
\end{definition}
If we consider the vertices in $V_j$ to have colour $j$, then a Hamiltonian
path or cycle in $K_{V_1, V_2, \dots, V_m}$ has the colouring property that
any two connected vertices in the path or cycle have different colours.
The problem of finding a Hamiltonian path in the multipartite graph is then
equivalent to the problem of arranging $|V_1|$
objects of colour $1$, $|V_2|$ objects of colour $2$, etc. in a linear
arrangement such that no two adjacent objects have the same colour.
Counting Hamiltonian paths is related to the following counting problem (CP1)
where, unlike the graph vertices, the objects of the same colour are
considered as indistinguishable.

\smallskip
\noindent
{\it
{\bf CP1.}
Let $m\in\mathbb N$, $m \ge 2$, and let $n_1, \dots, n_m\in\mathbb N$ be such
that $n_1, \dots, n_m \ge 2$.
Suppose we have $n_j$ indistiguishable objects of colour $j$, for all $j\in\{1, \dots, m\}$.
What is the number, denoted by $e_{(n_1, n_2, \dots, n_m)}$, of different ways
of arranging all objects in a line such that no two adjacent objects have
the same colour?
\/}

\smallskip
\noindent
This counting problem has been answered in the context of finding the number
of systems of sets of non-negative integers such that their Minkowski set sum
is equal to a given finite arithmetic progression;
the answer is \cite[Theorem 3]{LSDF&NSS}
\begin{align}
        e_n &=
\sum_{k_1=0}^{n_1-1} \sum_{k_2=0}^{n_2-1} \cdots \sum_{k_m=0}^{n_m-1} (-1)^{\sum_{j=1}^m k_j} {\sum_{j=1}^m (n_j - k_j) \choose (n_1-k_1),(n_2-k_2),\dots,(n_m-k_m)}
\prod_{j=1}^m {n_j - 1 \choose k_j}
\nonumber
\\
 &=
 \sum_{0_m \leq k \leq n-1_m}(-1)^{|k|}\binom{n-1_m}{k}\binom{|n-k|}{n-k}.
\label{eq:e}
\end{align}
In the second form of the above identity (and in the following) we use multi-index notation. For
$n, k \in \mathbb N_0^m$, we set
$|k| := \sum_{j=1}^m k_j$, $k! := \prod_{j=1}^m k_j!$, the (product of) binomial
coefficients
${n \choose k} := \frac{n!}{k!\,(n-k)!} = \prod_{j=1}^m {n_j \choose k_j}$ and
the multinomial coefficient
${|n| \choose n} := \frac{|n|!}{n!}$. Furthermore,
$0_m = (0, 0, \dots, 0) \in \mathbb N_0^m$,
$1_m = (1, 1, \dots, 1) \in \mathbb N_0^m$,
and
$k \le n$ if and only if $k_j \le n_j$ for all $j \in\{1, \dots, m\}$.

\bigskip
\section{Hamiltonian Paths in Complete Multipartite Graphs}

With the $e_n$ counting function in hand,
we now proceed to count the Hamiltonian paths in a complete multipartite graph.
To illustrate the connection in an example, consider the graph introduced in
the beginning.
We can indicate a directed Hamiltonian path by highlighting directed edges in
the graph in order, or simply by writing out the labels of the vertices in the
order in which they appear in the Hamiltonian path.

Here, for example is one possible Hamiltonian path,

    
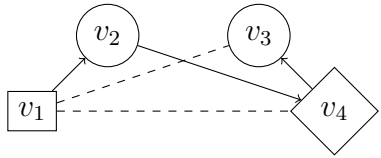
\begin{figure}[h]
    \centering
    \begin{tikzpicture}[scale=0.5]
        \node[rectangle,draw=black] (1) at (-1,-1) {$v_1$};
        \node[circle,draw=black] (2) at (1,1) {$v_2$};
        \node[circle,draw=black] (3) at (5,1) {$v_3$};
        \node[diamond,draw=black] (4) at (7,-1) {$v_4$};
        \draw[->] (1) -- (2);
        \draw[dashed,color=black] (1) -- (3);
        \draw[dashed,color=black] (1) -- (4);
        \draw[->] (2) -- (4);
        \draw[->] (4) -- (3);
    \end{tikzpicture}
    \end{figure}

    \FloatBarrier
    \noindent
or written as a vertex chain, $(v_1,v_2,v_4,v_3)$.

When we abstract from the individual labels of the vertices and only note the
symbol (or colour) that indicates which partite set the belong to, the above
Hamiltonian path is (unfaithfully) represented as follows.

\begin{figure}[h]
    \centering
    \begin{tikzpicture}[scale=0.5]
        \node[rectangle,draw=black] (1) at (-1,0) {};
        \node[circle,draw=black] (2) at (1,0) {};
        \node[diamond,draw=black] (3) at (3,0) {};
        \node[circle,draw=black] (4) at (5,0) {};
        \draw[color=black] (1) -- (2);
        \draw[color=black] (2) -- (3);
        \draw[color=black] (3) -- (4);
    \end{tikzpicture}
\end{figure}

To give a complete (faithful) representation of the Hamiltonian path, we can
add the vertex labels to the shaped (coloured) blocks, obtaining

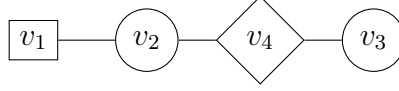
\begin{figure}[h]
    \centering
    \begin{tikzpicture}[scale=0.5]
        \node[rectangle,draw=black] (1) at (-2,0) {$v_1$};
        \node[circle,draw=black] (2) at (1,0) {$v_2$};
        \node[diamond,draw=black] (3) at (4,0) {$v_4$};
        \node[circle,draw=black] (4) at (7,0) {$v_3$};
        \draw[color=black] (1) -- (2);
        \draw[color=black] (2) -- (3);
        \draw[color=black] (3) -- (4);
    \end{tikzpicture}
\end{figure}
\FloatBarrier

The same arrangement of shaped (coloured) blocks also arises in the different
Hamiltonian path $(v_1, v_3, v_4, v_2)$,

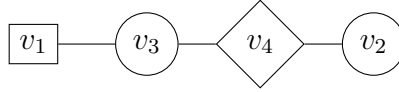
\begin{figure}[!h]
    \centering
    \begin{tikzpicture}[scale=0.5]
        \node[rectangle,draw=black] (1) at (-2,0) {$v_1$};
        \node[circle,draw=black] (2) at (1,0) {$v_3$};
        \node[diamond,draw=black] (3) at (4,0) {$v_4$};
        \node[circle,draw=black] (4) at (7,0) {$v_2$};
        \draw[color=black] (1) -- (2);
        \draw[color=black] (2) -- (3);
        \draw[color=black] (3) -- (4);
    \end{tikzpicture}
\end{figure}

\FloatBarrier

\noindent which is visually shown below.

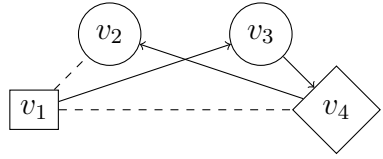
\begin{figure}[!h]
    \centering
    \begin{tikzpicture}[scale=0.5]
        \node[rectangle,draw=black] (1) at (-1,-1) {$v_1$};
        \node[circle,draw=black] (2) at (1,1) {$v_2$};
        \node[circle,draw=black] (3) at (5,1) {$v_3$};
        \node[diamond,draw=black] (4) at (7,-1) {$v_4$};
        \draw[dashed,color=black] (1) -- (2);
        \draw[->] (1) -- (3);
        \draw[dashed,color=black] (1) -- (4);
        \draw[<-] (2) -- (4);
        \draw[<-] (4) -- (3);
    \end{tikzpicture}
    \end{figure}

\FloatBarrier

\begin{theorem}\label{thm:Hpath}
    Let $K_{V_1,V_2,...V_m}$ be the complete multipartite graph with partite sets $V_1,V_2,...,V_m$. The number of directed Hamiltonian paths in $K_{V_1,V_2,...V_m}$ is
    \begin{align}\label{Thm2.2}
        H^P(K_{V_1,V_2,...V_m}) = e_{(|V_1|,|V_2|,...,|V_m|)}\prod_{i=1}^m|V_i|!,
    \end{align}

    \noindent where $e_n$ is given in Equation $(\ref{eq:e})$.
\end{theorem}

\begin{proof}
Every Hamiltonian path in the complete multipartite graph has consecutive vertices in different partite subsets,
since no two vertices in the same partite subset are connected by an edge.
It therefore corresponds to a
linear arrangement of $|V_1|, |V_2|, \dots, |V_m|$ otherwise indistinguishable
blocks of colour $1, 2, \dots, m$, respectively, such that no two consecutive
blocks share the same colour, and further an allocation of the $|V_j|$ vertices
in partite subset $V_j$ to the $|V_j|$ blocks of colour $j$.
There are $e_{(|V_1|, |V_2|, \dots, |V_m|)}$ different linear arrangements of
the blocks, and the number of different assignments of $|V_j|$ vertex labels
for colour $j$ is equal to $|V_j|!$, giving Equation (\ref{Thm2.2}) in
total.
\end{proof}

\bigskip
\section{Hamiltonian Cycles in Complete Multipartite Graphs}

We now turn to the question of counting the number of Hamiltonian cycles in
a complete multipartite graph. This is closely related to counting Hamiltonian
paths in such graphs, as every (directed) Hamiltonian cycle can be turned into
a Hamiltonian path by deleting one of its edges; the resulting path will have
the additional property that its first and last vertices lie in different
partite subsets. Conversely, every Hamiltonian path with the last property can
be converted into a Hamiltonian cycle by adding a closing edge between the
first and last vertices; this edge is available because the graph is complete
multipartite.
Note also that every directed Hamiltonian cycle gives rise to $|V|$ different
Hamiltonian paths as any one of its edges can be deleted.

Therefore the counting of Hamiltonian cycles can be reduced to counting
Hamiltonian paths with first and last vertex in different partite subsets.
This is related to the following counting problem.

\smallskip\noindent
{\it
{\bf CP2.}
Let $m\in\mathbb N$, $m \ge 2$, and let $n_1, \dots, n_m \in\mathbb N$ such
that $n_1, \dots, n_m \ge 2$.
Suppose we have $n_j$ indistinguishable objects of colour $j$, for all $j \in
\{1, \dots, m\}$.
What is the number, denoted by $\tilde e_{(n_1, n_2, \dots, n_m)}$, of different ways of arranging
all objects in a line such that no two adjacent objects have the same colour
and the first and last object have different colours?
\/}

\smallskip\noindent
In the following we prove a formula for $\tilde e_{(n_1, n_2, \dots, n_m)}$
using a collapsing argument analogous to the derivation of an expression
for $e_{(n_1, n_2, \dots, n_m)}$ in \cite{LSDF&NSS}.
As a starting point, we consider the following simpler counting problem
without the non-adjacency condition.

\smallskip\noindent
{\it
{\bf CP3.}
Let $m\in\mathbb N$, $m \ge 2$, and let $n_1, \dots, n_m \in\mathbb N$ such
that $n_1, \dots, n_m \ge 2$.
Suppose we have $n_j$ indistinguishable objects of colour $j$, for all $j \in
\{1, \dots, m\}$.
What is the number, denoted by $\tilde E_{(n_1, n_2, \dots, n_m)}$, of different ways of arranging
all objects in a line such that the first and last object have different colours?
\/}

\smallskip\noindent
\begin{lemma}\label{lem:E}
Let $m\in\mathbb N$, $m\ge 2$, and $n \in \mathbb N^m$. Then
$$
\tilde E_n = {|n| \choose n} \left(1 - \frac{\sum_{j=1}^m n_j\,(n_j-1)}{|n|\,(|n| - 1)} \right).
$$
\end{lemma}
\begin{proof}
We calculate $\tilde E_n$ by subtracting from the number of all different
linear arrangements the number of those arrangements that start and end with
an object of the same colour. For colour $j$, we place one object of this
colour at the beginning and one at the end of the line. (This presupposes
$n_j \ge 2$; otherwise there clearly is no linear arrangement both starting
and ending with an object of colour $j$.)
The remaining objects are placed between them in an unconstrained linear
arrangement of length $|n| - 2$, with only $n_j - 2$ objects of colour $j$.
Doing this for all colours, we find
\begin{align*}
\tilde E_n &= {|n| \choose n} - \sum_{j\in\{1, \dots, m\}, n_j > 1} {|n| - 2 \choose n_1, \dots, n_{j-1}, n_j - 2, n_{j+1}, \dots, n_m}
\\
&= {|n| \choose n} - \frac{(|n| - 2)!}{n_1!\,n_2!\,\cdots\,n_m!} \sum_{j\in\{1, \dots, m\}, n_j > 1} \frac{n_j!}{(n_j - 2)!}
\\
&= {|n| \choose n} \left(1 - \frac{(|n| - 2)!}{|n|!} \sum_{j\in\{1, \dots, m\}, n_j > 1} \frac{n_j!}{(n_j - 2)!} \right)
\\
&= {|n| \choose n} \left(1 - \frac{1}{|n|\,(|n| - 1)} \sum_{j=1}^m n_j\,(n_j - 1) \right).
\qedhere
\end{align*}
\end{proof}
\begin{remark}
The formula in Lemma \ref{lem:E} can also be written in the suggestive manner
\begin{equation*}
\tilde E_n = {|n| \choose n}  \frac{\left(\sum_{j=1}^m n_j\right)^2 - \sum_{j=1}^m n_j^2}{\left(\sum_{j=1}^m n_j\right)^2 - \sum_{j=1}^m n_j}\ .
\end{equation*}
\end{remark}

\begin{theorem}\label{thm:Ee}
Let $m \in \mathbb N$, $m \ge 2$, and $n\in\mathbb N^m$. Then
\begin{equation*}
\tilde e_n = \sum_{0_m \le k \le n - 1_m} (-1)^{|k|} {n - 1_m \choose k} \tilde E_{n-k}.
\end{equation*}
\end{theorem}
\begin{proof}
First consider the linear arrangements of $n_1, \dots, n_m$ objects of colours
$1, \dots, m$, for general $n \in\mathbb N^m$, only under the constraint that
the first and last objects have different colours.
The number of such arrangements is $\tilde E_n$, given in Lemma \ref{lem:E}.
When we collapse such an arrangement by replacing every consecutive chain of
objects of the same colour by a single object of that colour, we obtain a
linear arrangement of (fewer) objects that satisfies the conditions that no
two adjacent objects have the same colour and that the first and last objects
have different colours.
If the collapsed arrangement has $k_j$ objects of colour $j$, then there are
${n - 1_m \choose k - 1_m}$ different ways this collapsed line can have arisen
from an original arrangement of $n_1, \dots, n_m$ objects; indeed, the
assignment of $n_j$ objects to $k_j$ collapsed objects is equivalent to
placing $k_j-1$ dividers in the $n_j - 1$ spaces between the $n_j$ objects.

We thus obtain the equation
\begin{equation}\label{eq:Ee}
\tilde E_n = \sum_{1_m \le k \le n} {n - 1_m \choose k - 1_m} \tilde e_k.
\end{equation}
Now we use the multivariate power series
\begin{equation*}
\exp\left(\pm\sum_{j=1}^m x_j \right) = \sum_{l \in \mathbb N_0^m} (\pm 1)^{|l|} \frac{x^l}{l!}
\qquad (x = (x_1, x_2, \dots, x_m) \in \mathbb C^m)
\end{equation*}
and the Cauchy product formula to write the generating function for
$\tilde e_k$ as follows,
\begin{align*}
 \sum_{k \in {\mathbb N}_0^m} \tilde e_{1_m+k}\,\frac{x^k}{k!}
 &= \left(\sum_{k \in {\mathbb N}_0^m} \tilde e_{1_m+k}\,\frac{x^k}{k!} \right)
  \left(\sum_{l \in {\mathbb N}_0^m} \frac{x^l}{l!} \right)
  \left(\sum_{l \in {\mathbb N}_0^m} (-1)^{|l|} \frac{x^l}{l!} \right)
\\
 &= \left(\sum_{n \in {\mathbb N}_0^m} \left(\sum_{0_m \le k \le n} \frac{\tilde e_{1_m+k}}{k! (n-k)!} \right) x^n \right)
  \left(\sum_{l \in {\mathbb N}_0^m} (-1)^{|l|} \frac{x^l}{l!} \right)
\\
 &= \left(\sum_{n \in {\mathbb N}_0^m} \left(\sum_{0_m \le k \le n} {n \choose k} \tilde e_{1_m+k} \right) \frac{x^n}{n!} \right)
  \left(\sum_{l \in {\mathbb N}_0^m} (-1)^{|l|} \frac{x^l}{l!} \right)
\\
 &= \left(\sum_{n \in {\mathbb N}_0^m} \tilde E_{n + 1_m} \frac{x^n}{n!} \right)
  \left(\sum_{l \in {\mathbb N}_0^m} (-1)^{|l|} \frac{x^l}{l!} \right)
\\
 &= \sum_{k \in {\mathbb N}_0^m} \left(\sum_{0_m \le l \le k} (-1)^{|l|}\,\frac{k!}{l! (k-l)!}\,\tilde E_{1_m+k-l} \right) \frac{x^k}{k!}.
\end{align*}
(Note that the power series defining the generating function is convergent for
all $x\in\mathbb C^m$) as $\tilde e_{1_m + k} \le e_{1_m + k}$ and the
series for the generating function of the latter was shown to be convergent in
\cite{LSDF&NSS}.)
Hence we can read off
\begin{equation*}
\tilde e_{1_m + k} = \sum_{0_m \le l \le k} (-1)^{|k|} {k \choose l} \tilde E_{1_m + k - l},
\end{equation*}
which gives the identity stated in the theorem.
\end{proof}
\begin{remark}
The proof in \cite{LSDF&NSS} is fully analogous to the above, with the difference that the left-hand side in
equation (\ref{eq:Ee}) is replaced with the count of
unconstrained linear arrangements, i.e.\ the multinomial ${|n| \choose n}$.
This suggests that the collapsing argument used here may be applicable to
other situations where a colour non-adjacency condition is to be introduced
into a counting argument.
\end{remark}
\smallskip\noindent
As the vertices in each partite subset are distinguishable and hence subject to
permutation in the Hamiltonian cycle, we obtain the following counting formula.
\begin{theorem}\label{thm:cyc}
The number of directed Hamiltonian cycles in the complete multipartite graph
$K_{V_1, V_2, \dots, V_m}$ is equal to
$$
 H^C(K_{V_1, V_2, \dots, V_m}) = \tilde e_{(|V_1|, |V_2|, \dots, |V_m|)}\,
 \frac{\prod_{j=1}^m |V_j|!}{|V|}\,.
$$
\end{theorem}
\begin{proof}[Proof of Theorem \ref{thm:cyc}.]
Every directed Hamiltonian cycle gives rise to $|V|$ different directed
Hamiltonian paths with the property that the start and end vertex lie in
different partite subsets; therefore we can reduce the question to counting
the latter.
In analogy to the proof of Theorem \ref{thm:Hpath}, we observe that any such
Hamiltonian path corresponds to a linear arrangement of coloured blocks as
in counting problem CP2 and subsequent assignment of vertices in each
partite subset to blocks of the corresponding colour, so there are
$\tilde e_{(|V_1|, |V_2|, \dots, |V_m|)} \prod_{j=1}^m |V_j|!$ such directed Hamiltonian paths.
\end{proof}
\begin{remark}
The formula in Theorem \ref{thm:cyc} is intriguing in that the fraction on the right hand side does
not in general produce an integer by itself; e.g. for $|V_1| = 3$, $|V_2| = |V_3| = 2$
we obtain by Theorem \ref{thm:Ee}
\begin{equation*}
H^C(K_{V_1, V_2, V_3}) =
\tilde e_{(3,2,2)}\,\frac{3!\,2!\,2!}{7} = 28\,\frac{24}{7} = 96.
\end{equation*}
The total result is of course an integer (in fact an even integer since every directed Hamiltonian cycle is paired with the same cycle in opposite direction), but there is a subtle cancellation
between the 7 in the denominator and the value of $\tilde e_{(3,2,2)}$.

By the explicit formulae found in Lemma \ref{lem:E} and Theorem \ref{thm:Ee},
the count of directed Hamiltonian cycles in $K_{V_1, V_2, \dots, V_m}$ takes the full form,
written without multi-index notation for easier comparison,
\begin{align*}
\frac{\prod_{j=1}^m |V_j|!}{|V|} &\sum_{k_1=0}^{|V_1|-1} \sum_{k_2=0}^{|V_2|-1} \cdots \sum_{k_m = 0}^{|V_m|-1} (-1)^{\sum_{j=1}^m k_j} {|V_1| - 1 \choose k_1}
{|V_2| - 1 \choose k_2} \cdots {|V_m| - 1 \choose k_m}
\\
&\qquad\times
{\sum_{j=1}^m (|V_j| - k_j)
\choose |V_1|-k_1, \dots, |V_m|-k_m}
\left(1 - \frac{\sum_{j=1}^m (|V_j|-k_j)(|V_j|-k_j-1)}{(|V| - \sum_{j=1}^m k_j)(|V|-1 - \sum_{j=1}^m k_j)}\right),
\end{align*}
which is similar to, but not quite identical with the formula in \cite{MSmirnov}.
\end{remark}



\end{document}